\documentclass[11pt]{article}
\usepackage{amssymb,amsmath}
\usepackage{amsthm,,latexsym}
\usepackage{graphicx}
\usepackage{enumerate}

\numberwithin{equation}{section}

\newtheorem{lemma}{Lemma}[section]
\newtheorem{theorem}[lemma]{Theorem}
\newtheorem{corollary}[lemma]{Corollary}

\newtheorem{proposition}[lemma]{Proposition}

\theoremstyle{definition}
\newtheorem{example}[lemma]{Example}

\newtheorem{remark}[lemma]{Remark}

\newtheorem{definition}[lemma]{Definition}

\long\def\symbolfootnote[#1]#2{\begingroup%
\def\thefootnote{\fnsymbol{footnote}}\footnote[#1]{#2}\endgroup}

\begin{document}
\newcommand{\ul}{\underline}
\newcommand{\be}{\begin{equation}}
\newcommand{\ee}{\end{equation}}
\newcommand{\ben}{\begin{enumerate}}
\newcommand{\een}{\end{enumerate}}

\newcommand{\con}{\mathbin\Vert}
\newcommand{\R}{\mathbb{R}}
\newcommand{\N}{\mathbb{N}}
\newcommand{\Z}{\mathbb{Z}}
\newcommand{\Id}{\mathbb{I}_d}
\newcommand{\J}{\mathcal{J}}
\newcommand{\mL}{\mathcal{L}}
\newcommand{\A}{\tilde{A}}
\newcommand{\al}{\alpha}
\newcommand{\bb}{\beta}

%\long\def\symbolfootnote[#1]#2{\begingroup%
%\def\thefootnote{\fnsymbol{footnote}}\footnote[#1]{#2}\endgroup}
%\textwidth=30cc
%\baselineskip = 16pt
\noindent

\title{Multifactorisations and Divisor Functions}
%\date{Revised February 17, 2011}
\date{\today}

\author{A. D. Law,\quad M. C. Lettington\quad and\quad K. M. Schmidt \vspace{5mm} \\ Cardiff University, School of Mathematics, UK}

\maketitle
\begin{abstract}
We consider a joint ordered multifactorisation for a given positive integer $n\geq 2$ into $m$ parts, where $n=n_1~\times~\ldots~\times~n_m$, and each part $n_j$ is split into one or more component factors. Our central result gives an enumeration formula for all such joint ordered multifactorisations $\mathcal{N}_m(n)$. As an illustrative application, we show how each such factorisation can be used to uniquely construct and so count the number of distinct additive set systems
(historically referred to as complementing set systems). These set systems under set addition generate the first $n$ non-negative consecutive integers uniquely and, when each component set is centred about 0, exhibit algebraic invariances. For fixed integers $n$ and $m$, invariance properties for $\mathcal{N}_m(n)$ are established. The formula for $\mathcal{N}_m(n)$ is comprised of sums over associated divisor functions and the Stirling numbers of the second kind, and we conclude by deducing sum over divisor relations for our counting function $\mathcal{N}_m(n)$.
\end{abstract}

\section{Introduction}
\symbolfootnote[0]{
2010 \emph{Mathematics Subject Classification}: 11B13, 11E25, 11B25, 11B30, 11A51.\newline
\emph{Key words and phrases}: multifactorisations, joint ordered factorisations, divisor functions, additive systems, algebraic invariances.}
We begin by outlining two of the three main areas under consideration, multifactorisations and divisor functions, before stating our main results. These relate to the counting function $\mathcal{N}_m(n)$, which counts the number of joint ordered factorisations of the positive integer $n$ into $m$ parts. \vspace{3 mm}

\noindent
In 1893 P. A. MacMahon discussed in his \emph{Memoir on the Composition of Numbers} \cite{macmahon}
the arithmetic function (our notation)
$c(n)$, which gives the number of ordered factorisations of $n$ into factors greater than~1.

\begin{example} When $n=12$ the number of ordered factorisations $c(12)$ is given by counting
\begin{align*}
12&=12
=2\times 6 = 6\times 2=3\times 4 = 4\times 3
=2\times 2\times 3 = 2\times 3 \times 2 = 3\times 2\times 2,\\
&\Rightarrow c(12)=1+4+3=8.
\end{align*}
\end{example}
%E. C. Titchmarsh considers $c(n)$ in the early pages of his famous book \emph{The Theory of the Riemann Zeta Function} %(1951), setting $c(1)=1$,
%with Dirichlet series for $\Re s>1$
%\[
%\sum_{n=1}^\infty \frac{c(n)}{n^s}=\frac{1}{2-\zeta(s)},\quad\text{where}\quad \zeta{(s)}=\sum_{n=1}^\infty \frac{1}{n^s}.
%\]
In \cite{mcl2,mnh1} concise notation for the concept of an \emph{m-part joint ordered factorisation} was introduced, as defined below, where we use the notation
$\mathbb{N}_{\geq 2}=\mathbb{N}+\{1\}=\{2,3,4,\ldots\}$.
\begin{definition}[joint ordered factorisation] Let $m\in\mathbb{N}$, $a=(n_1,\dots,n_m)\in \mathbb{N}_{\geq 2}^m$, with $n=n_1\times\ldots \times n_m$. Then we call
\begin{equation*}
    \Big(\big(j_1,f_1\big),\big(j_2,f_2\big),\dots,\big(j_L,f_L\big)\Big)\in\big(\{1,2,\dots,m\}\times(\mathbb{N}_{\geq 2})\big)^L,
\end{equation*}
where $L\in\mathbb{N}$, an \textit{m-part joint ordered factorisation of $n$} if $j_\ell\neq j_{\ell-1}$ $\big(\ell\in\{2,\dots,L\}\big)$ and
\begin{equation*}
    \prod_{j_\ell= j}f_\ell=n_j,\quad\text{for all}\quad \big(j\in\{1,\dots,m\}\big),
\end{equation*}
where
%ere $\mathcal{L}_j:=\{\ell\mid j_\ell=j\}=\{\ell_1^{(j)},\dots,\ell_{k_j}^{(j)}\}$, with suitable $k_j\in\N$, so that
$\ell$ indexes the positions of the tuples in the joint ordered factorisation that correspond to factors of $n_j$.

Furthermore we define the partial products of all factors before the $\ell$-th
%and of the factors with index class $j$ less that $\ell$,
in the joint ordered factorisation, as
\[
F(\ell)=\prod_{s=1}^{\ell-1}f_s,\quad\text{so that}\quad n=\prod_{j=1}^m n_j=F(L+1),
%\quad P_j(\ell)=\prod_{\substack{q\in \mathcal{L}_j \\ q<\ell}}f_q,
\]
where for $\ell=1$ we use the usual convention for the empty product $F(1)=1$.
%respectively, so that
%\[
%N=\prod_{j=1}^m n_j=F(L+1),\quad\text{and}\quad F(\ell)=\prod_{j=1}^mP_j(\ell).
%\]
%Furthermore, define $e_j:=\max\{\kappa\in\{1,\dots,k_j\}\mid f_{l_\kappa^{(j)}}\text{ even}\}$ with the $j$-th index %subset $\mathcal{L}_j':=\{\ell\mid j_\ell=j,\ell>\ell_{e_j}\}$. If $n_j$ is odd, then $e_j:=0$, %$\mathcal{L}'_j=\mathcal{L}_j$ and we say $\ell_{e_j}=f_{\ell_{e_j}}=F(\ell_{e_j})=P_j(\ell_{e_j})=0$.
\end{definition}
In other words, a joint ordered factorisation of an $n$-tuple of natural numbers
$a=(n_1, \dots, n_m)$ arises from writing each of these numbers as a product of non-trivial factors,
i.e., factors\ $\ge 2$, and then arranging all factors in a linear chain such
that no two adjacent factors arise from the factorisation of the same number.
Denoting by $\Omega(n)$ the total number of prime factors of $n$ including repeats, it follows from the above definition that the maximum length $L$ (number of tuples) in the joint ordered factorisations for $n$ into $m$ parts
satisfies $L\leq \Omega(n)$.

In \cite{rOB} Ollerenshaw and Br\'ee considered the two-part joint ordered factorisations as divisor paths.

\begin{example} When $n=12$, and the number of parts in the factorisation is $m=2$, we have $n=12=n_1\times n_2$. Working through each possibility, we find 14 2-part joint ordered factorisations, 7 of which are given below with the partial product progressions for the $a=(n_1,n_2)$ tuples.
\begin{align*}
&\text{\bf Ordered Factorisation} & \qquad \qquad {\bf a=(n_1,n_2)} &\qquad\qquad &\quad\text{\bf Partial Products $F$}&\\
&\left ((1,2)(2,6)\right )& (2,6)\quad &\qquad\qquad& 1,2,12\qquad\qquad\quad\,\,\,&\\
&\left ((1,6)(2,2)\right )& (6,2)\quad  &\qquad\qquad& 1,6,12\qquad\qquad\quad\,\,\,&\\
&\left ((1,3)(2,4)\right ) & (3,4)\quad  &\qquad\qquad& 1,3,12\qquad\qquad\quad\,\,\,&\\
&\left ((1,4)(2,3)\right ) & (4,3)\quad &\qquad\qquad& 1,4,12\qquad\qquad\quad\,\,\,&\\
&\left ((1,2)(2,3)(1,2)\right ) & (4,3)\quad  &\qquad\qquad& 1,2,6,12\qquad\qquad\,\,\,&\\
&\left ((1,3)(2,2)(1,2)\right )& (6,2)\quad  &\qquad\qquad& 1,3,6,12\qquad\qquad\,\,\,&\\
&\left ((1,2)(2,2)(1,3)\right )& (6,2)\quad  &\qquad\qquad& 1,2,4,12\qquad\qquad\,\,\,&\\
\end{align*}
The remaining 7 possibilities are then obtained by starting the joint ordered factorisations from $n_2$ rather than $n_1$, so that ((1,2)(2,2)(1,3)) becomes ((2,2)(1,2)(2,3)) etc.

Repeating the above constructions when there are $m=1$ and $m=3$ parts, we find that there are respectively $1\times 1!=1$
and $3\times 3!= 18$ distinct joint ordered factorisations, so in total $1+14+18=33$ factorisations for $12$. As $\Omega(12)=3$, no factorisations exist for $m\geq 4$.
\end{example}
MacMahon considered the prime factor exponents of the positive integer 
$n = p_1^{a_1} p_2^{a_2} \cdots p_\nu^{a_\nu}\geq 2$,
where $p_1, p_2, \dots, p_\nu$ are distinct prime numbers. He obtained the formula for $c(n)$,
%Let
%$j \in {\mathbb N}.$ Then for $n = p_1^{a_1} p_2^{a_2} \cdots p_\nu^{a_\nu}$ with distinct primes
%$p_1, p_2, \dots, p_\nu,$ and setting $\Omega(n)=a_1+a_2+\ldots +a_\nu$, the number of prime factors $n$ including %repeats, we have
\be\label{eq:cn}
c(n)=\sum_{j=1}^{\Omega(n)}c_j(n) =\sum_{j=1}^{\Omega(n)} \sum_{k=0}^j (-1)^k \binom{j}{k} \prod_{\ell=1}^\nu \binom{a_\ell +  j - k - 1}{a_\ell}
\ee
%which can be written as (see Lemmas 1 and 4 of \cite{mcl2})
%\be\label{eq:cjn}
%c(n) =\sum_{j=1}^{\Omega(n)}c_j(n),\quad\text{where for $j\geq 1$},\quad c_j(n)= \sum_{k=0}^j (-1)^k \binom{j}{k} %\prod_{\ell=1}^\nu \binom{a_\ell +  j - k - 1}{a_\ell}.
%\ee
(see \cite[\S10, p.843]{macmahon}, \cite[eq.(4)]{survey}),
where $c_j(n)$ is the number of ways of writing $n$ as an ordered product of $j$ non-trivial factors.
% so for $j\geq 1$,
%$c_j(n)$ counts the number of ways of writing $n=n_1\times \ldots \times n_j$, where different orderings are counted.
We call $c_j(n)$ the $j$-th non-trivial divisor function.
Eq.\ (\ref{eq:cn}) was rediscovered (see Lemmas 1 and 4 of \cite{mcl2}) through the theory of $j$-th divisor functions \cite{mcl4}, obtaining in the process the associated divisor function $ c_j^{(r)}(n)$, defined for
$j \in {\mathbb N}_0,$
$r \in {\mathbb Z},$
such that
\be\label{eq:cjrn}
 c_j^{(r)}(n) = \sum_{k=0}^j (-1)^k \binom{j}{k} \prod_{\ell=1}^\nu \binom{a_\ell + r + j - k - 1}{a_\ell},
\ee
where we take $c_j^{(0)}=c_j$ with $c_0(n)=1$ if $n=1$ and $c_0(n)=0$ if $n\geq 2$. The difference between the two formulae
in (\ref{eq:cjrn}) and the inner sum in (\ref{eq:cn}) is subtle, with $j$ replaced with $j+r$ in the inner binomial coefficient in (\ref{eq:cn}), whilst the outer binomial coefficient remains unchanged.

Results in \cite{mcl4} show that
$|c_j^{(r)}(n)|$ counts the number of factorisations of $n$ into $j+r$ factors, the first $j$ being non-trivial (i.e. factors $\geq 2$)
if $r \ge 0$, and into $-r$ factors, of which the first
$j$ factors are non-trivial and all factors must be square-free if $r< 0$ and $j<|r|$.
If $r< 0$ and $j>|r|$, then cancellations occur in the formula for $c_j^{(r)}(n)$ and a combinatorial interpretation cannot be inferred.
%into $\max\{j, -r\}$ factors if $r < 0,$ of which at least
%$j$
%must be non-trivial, so again we find $c_j^{(r)}(n) = 0$ if $j > \Omega(n).$
%(Also, if $r < 0$, then at least $-r$ factors must be square-free.)
%\par
The special case
$j = -r$
%\begin{align}
% c_j^{(-j)}(n) &= (-1)^j \sum_{n_1 n_2 \cdots n_j = n} (\mu-e)(n_1)\, (\mu-e)(n_2) \cdots (\mu-e)(n_j) \qquad (n \in %{\mathbb N}),
%\label{especial}\end{align}
turns out to be of particular importance (cf. \cite{mcl4}\ Theorem 4), as
the value of
$c_j^{(-j)}(n)$
can be interpreted as
$(-1)^{\Omega(n) + j}$
times the number of ordered factorisations of
$n$
into $j$ non-trivial, square-free factors. For more information about these functions, see \S 2 below.

For a fixed $m$-tuple $a=(n_1,\ldots,n_m)$, with $n=n_1\times \ldots \times n_m$, using the associated divisor function $c_j^{(-j)}(n)$, it was proven (\cite{mcl4} Theorem 4) that the total number of different $m$-part joint ordered factorisations for $n$ over the $m$-tuple $a$ is given by
\be\label{eq:tuplecount}
M_a(n)=\sum_{\ell \in {\mathbb N}^m} \binom{|\ell|}{ \ell} \prod_{j=1}^m c_{\ell_j}^{(-\ell_j)}(n_j),
\ee
where in the usual notation $\binom{|\ell|}{ \ell}$ denotes the multinomial coefficient for $\ell=(\ell_1,\ldots,\ell_m)$, with
$|\ell|=\ell_1+\ldots +\ell_m$.
% so that
%\[
%\binom{|\ell|}{ \ell}=\frac{(\ell_1+\ldots +\ell_m)!}{\ell_1!\times \ell_2!\times\ldots\times\ell_m!}.
%\]
Although concisely notated, the above counting formula is in practice lengthy to use for large values of $m$ and $\Omega(n)$, as it requires that the sum and product runs over all $m$-tuples $\ell$ with $1\leq \ell_j\leq \Omega(n_j)$.

However, the formula (\ref{eq:tuplecount}) partially motivates this paper, being central to deriving our main result for
\[
\mathcal{N}_m(n)=
\sum_{\substack{a \in {\mathbb N}_{\geq 2}^m\\ \prod n_j =n}}M_a(n),
\]
which counts the number of all $m$-part joint ordered factorisations of $n$. It employs
the modified M\"obius function, defined below.

\begin{definition}[modified M\"obius function]\label{defn4}
Denote by  $\mu(n)$ the M\"obius function, returning $(-1)^{\Omega(n)}$ if $n$ is square free, $0$ otherwise, and by $e(n)$ the convolution identity, returning $1$ if $n=1$ and $0$ if $n\geq 2$. Then the \emph{modified M\"obius function} $(\mu-e)(n)$ is defined by
\[
 (\mu - e)(n) = \begin{cases} (-1)^{\Omega(n)} & \text{\rm if $n$ is square-free} \\
  0 &  \text{\rm otherwise (including the case $n = 1$)}
  \end{cases}
 \qquad (n \in {\mathbb N}).
\]
\end{definition}
\begin{theorem}[factorisation enumeration theorem]\label{enum22}
Denote by $S(L,m)$ the Stirling numbers of the second kind, by $(\mu-e)^{*L}$ the $L$-th convolution of the modified
M\"obius function, and by $c_L^{(-L)}(n)$ the special case of the associated divisor function.

Then for natural numbers
$m,n \in \mathbb{N}$,
the number of $m$-part joint ordered factorisations of $n$
is equal to
\[
 {\cal N}_m(n) = \sum_{L=0}^{\Omega{(n)}} (-1)^L m!\, S(L,m)\, (\mu-e)^{*L}(n)
 = \sum_{L=0}^{\Omega{(n)}} m!\, S(L,m)\, c_L^{(-L)}(n).
\]
\end{theorem}
Here the well known Stirling numbers of the second kind $S(L,m)$ \cite{riordan} count the number of ways to partition a set of $L$ objects into $m$ nonempty subsets.

We will prove Theorem \ref{enum22} in \S 3. For $m=2$ we also have the next result. 

\begin{theorem}\label{thm55}
For $n\in\mathbb{N}$
\begin{equation}\label{sdf12}
\mathcal{N}_2(n)=2\sum_{L=2}^{\Omega{(n)}} c_L(n).
\end{equation}

%so that
%\begin{equation}\label{sdf3}
%{\cal M}_2(n)=\mathop{\sum_{d\mid n}}_{d<n}\left ( {\cal M}_2(d)+{\cal M}_{1}(d) \right )
%=c_2(n)+\mathop{\sum_{d\mid n}}_{d<n} {\cal M}_2(d).
%\end{equation}
%For $\Re s>1$ the counting function ${\cal N}_2(N)$ has Dirichlet series
%\begin{equation}\label{sdf3}
%\sum_{n=1}^\infty \frac{{\cal N}_2(n)}{n^s}=2\sum_{j=2}^\infty (\zeta(s)-1)^j
%=2\sum_{j=2}^\infty  \sum_{k=0}^j(-1)^{j+k}\binom{j}{k}\zeta(s)^k,
%\end{equation}
%which for real number $\alpha\approx 1.72865$, so that $\zeta(\alpha)=2$, can be written for $\Re s>\alpha$ as
%\begin{equation}\label{sdf33}
%\frac{1}{2}\sum_{n=1}^\infty \frac{{\cal N}_2(n)}{n^s}=\frac{1}{2-\zeta(s)}-\zeta{(s)}=\frac{(\zeta(s)-1)^2}{2-\zeta(s)}.
%\end{equation}
\end{theorem}
\begin{proof}
For $m=2$, any joint ordered factorisation of $n$ arises from taking $f_1, \dots, f_L$ to be the factors in an ordered factorisation of $n$ into at least 2 non-trivial factors and $j_1, \dots, j_L$ to alternate between 1 and 2, starting either with $j_1=1$ or with $j_1=2$.
Therefore the total number of joint ordered factorisations is twice the number of ordered factorisations into $L$ non-trivial factors, where $L$ ranges between 2 and $\Omega(n)$, and hence the result (\ref{sdf12}).
\end{proof}
The counting function $\mathcal{N}_m(n)$ satisfies the following sum over divisors relations.

\begin{theorem}[sum over divisors theorem]\label{sdf}
Let $m\in \mathbb{N}$, and $n\in \mathbb{N}_{\geq 2}$ with $\mu(n)$ the M\"obius function, ${\cal N}_m(n)$ the counting function introduced in Theorem \ref{enum22} for the number of $m$-part joint ordered factorisations of $n$. Then
${\cal N}_m(n)$ obeys the sum over divisors relations
\begin{align}\label{sdf1}
{\cal N}_m(n)&=\mathop{\sum_{d\mid n}}_{d<n}\left ((m-1) {\cal N}_m(d)+m{\cal N}_{m-1}(d) \right )\\
\label{sdf2}
&= -m \mathop{\sum_{d\mid n}}_{d<n}\mu\left (\frac{n}{d}\right)\left ({\cal N}_m(d)+{\cal N}_{m-1}(d) \right ).
\end{align}
For $m=2$, equation (\ref{sdf1}) can be written as
\begin{equation}\label{sdf3}
{\cal N}_2(n)
%=\mathop{\sum_{d\mid n}}_{d<n}\left ( {\cal N}_2(d)+2{\cal N}_{1}(d) \right )
=2d_2(n)-4+\mathop{\sum_{d\mid n}}_{d<n} {\cal N}_2(d)
=2c_2(n)+\mathop{\sum_{d\mid n}}_{d<n} {\cal N}_2(d).
\end{equation}
\end{theorem}

%\begin{remark}
%For ${\cal N}_m(n)$ we will get the sum $m!\sum_{j=m}^\infty a_j(n)c_j(n)$, where the $a_j$ count the number of ways of %reordering the factors in the joint ordered factorisations, whilst maintaining non-consecutive factors from the same part.
%\end{remark}
\begin{example}\label{exam4}
When $n=12$, we find that ${\cal N}_2(12)=14$, whereas $2d_2(12)-4=8$, and
\[
\mathop{\sum_{d\mid 12}}_{d<12}{\cal N}_2(12)=0+0+0+2+4=6,
\]
so that as given by the formula in (\ref{sdf3}), $14=8+6$.

%For $N=p$ a prime number, we have that $2d_2(p)-4=4-4=0={\cal N}_2(p)={\cal N}_2(1)$, and both sides equate to zero.
\end{example}

%With the topics of divisor functions and joint ordered factorisations introduced, and the main results for our
%enumeration function $\mathcal{N}_m(n)$ stated, to give a structural overview, we next
The remainder of the paper is structured as follows. In \S 2 we prove Theorems \ref{enum22} and \ref{sdf}.
As an application
of our enumeration function $\mathcal{N}_m(n)$, we consider in \S 3 the related topic of sum systems, deducing using a bijection that $\mathcal{N}_m(n)$ also counts the number of $m$-part sum systems for $n$. A sum system under set addition generates the first $n$ non-negative integers uniquely (without repeats) and our new results here give algebraic invariance properties for fixed values of $n$ and $m$ whilst varying the joint ordered factorisation constructions.
%count the number of  an alternative statement in terms of divisor functions for (in our notation) $\mathcal{N}_2(n)$ the %two-dimensional considerations of Long \cite{long}, the correct form of which was provided by MacMahon in \cite{macmahon}
%(see p5 of \cite{survey}).
In \S 4 we summarise our findings and consider future areas for exploration.

\section{Divisor functions, Dirichlet convolutions and proofs}

To prove Theorem \ref{sdf}, we first need to understand the associated divisor function in terms of Dirichlet convolutions.
Our starting point is the commutative Dirichlet convolution algebra of arithmetic functions, where the convolution of arithmetic functions
$f_1, f_2, \dots, f_j$
is given by
\begin{align}
 (f_1 * f_2 * \cdots * f_j)(n) &= \sum_{n_1 n_2 \cdots n_j = n} f_1(n_1) f_2(n_2) \cdots f_j(n_j),
\label{econvo}\end{align}
summing over all ordered factorisations of
$n \in {\mathbb N}$
into
$j$
factors.
We denote the $j$th convolution power as follows,
$f^{*j} := f * f * \cdots * f,$
where the right-hand side has
$j$
repetitions of
$f;$
by the usual convention,
$f^{*0} = e.$
The function
$e(n) = \delta_{n, 1}$
$(n \in {\mathbb N})$
is the neutral element of the Dirichlet convolution product, and the
convolution inverse of the constant function 1 is the well-known M\"obius
function~$\mu$, both being introduced in Definition \ref{defn4}.
%In order to state our results we first need to introduce some arithmetic divisor functions.
\par
The $j$th convolution of the constant function 1 is more generally known as the classical $j$th divisor function
$d_j = 1^{*j}$
(cf.\ \cite[p.~9]{rSchSp}),
which counts the ordered factorisations of its argument into
$j$
positive integer factors.
It can be shown (see \S 2 of \cite{mcl2}) that $d_j$ satisfies the sum-over-divisors recurrence relation
\begin{equation}\label{sodd}
d_{j+1}(n) = (d_j * 1)(n) = \sum_{m|n}d_j(m)= 1^{*j+1}(n),\qquad (n,j \in\mathbb{N})
\end{equation}
and has the Dirichlet series \cite{ect}
\[ \sum_{n=1}^\infty \frac{d_j(n)}{n^s} = \zeta(s)^j,\quad\text{where}\quad
\zeta(s)=\sum_{n=1}^\infty \frac{1}{n^s},\quad \Re s>1.
 \]
%These statements extend to the case $j = 0$ if we set $d_0(n) = e(n)$, with $e$ defined as above.
In contrast, the $j$th non-trivial divisor function $c_j$ only counts ordered factorisations in which all
factors are greater than 1. It can be
expressed as the $j$-fold Dirichlet convolution
$c_j = (1 - e)^{*j}$,
and so satisfies the slightly different sum-over-divisors recurrence relation
%compared to (\ref{sodd}),
\begin{align*}
 c_{j+1}(n) &= (c_j * (1 - e))(n) = \sum_{m | n} c_j(m) \,(1-e)\left(\frac n m\right)
 = \sum_{\substack{m|n\\ m<n}} c_j(m), %= \mathop{\sum_{m|n}}_{m\notin\{1,n\}} c_j(m)
\qquad (n, j \in \mathbb{N}).
%\label{sodc}
\end{align*}
As the Dirichlet series for $1 - e$ is $\zeta(s) - 1$, the non-trivial divisor
function $c_j$ has the Dirichlet series
	\[\sum_{n=1}^{\infty} \frac{c_j(n)}{n^s} = \left(\zeta(s)-1\right)^j. \]
These formulae
extend to $j = 0$ when we set $c_0 = e\ = d_0$.
It is important to note that $c_j$, unlike $d_j$, is \emph{not} a multiplicative arithmetic function.

Combining these two functions yields
the
{\it associated $(j, r)$-divisor function\/},
defined for non-negative integers
$r,j$ as
$c_j^{(r)} = (1-e)^{*j} * 1^{*r}.$
It was demonstrated in \cite{mcl4} that
this arithmetic function counts the ordered factorisations of its argument into
$j+r$
factors, of which the first
$j$
must be non-trivial.
\par
Moreover, as the constant function 1 has a convolution inverse, this definition extends
naturally to negative upper indices, giving the associated $(j, -r)$-divisor
function
$c_j^{(-r)} = (1-e)^{*j} * \mu^{*r}.$
We note that
$(1-e)$
does not have a convolution inverse, as $(1-e)(1)=0$, so there is no analogous extension to
negative lower indices.

Popovici \cite{rPopo} studied the functions
$c_0^{(-r)} = \mu^{*r}$.
In the associated $(j, -r)$-divisor functions, the
modified M\"obius function of Definition \ref{defn4}
%\[
 %(\mu - e)(n) = \begin{cases} (-1)^{\Omega(n)} & \text{\rm if $n$ is square-free} \\
  %0 &  \text{\rm otherwise (including the case $n = 1$)}
  %\end{cases}
 %\qquad (n \in {\mathbb N}),
%\]
appears naturally. It was also shown in \cite{mcl4} that if
$j \geq r,$ then the Dirichlet convolutions have the forms
\begin{align}
 c_j^{(-r)} &= (1-e)^{*j-r} * ((1-e) * \mu)^{*r} = (-1)^r (1-e)^{*j-r} * (\mu-e)^{*r};
\nonumber\end{align}
if
$j < r,$
then
\begin{align}
 c_j^{(-r)} &= ((1-e) * \mu)^{*j} * \mu^{*r-j} = (-1)^j (\mu-e)^{*j} * \mu^{*r-j},
\nonumber\end{align}
%so that
%$c_j^{(r)}(n)$
%involves factorisation of
%$n$
%into
%$j+r$
%factors if
%$r \geq 0,$
%into
%$\max\{j, -r\}$
%factors if
%$r < 0,$
%of which at least
%$j$
%must be non-trivial. Also for $r < 0$, at least $-r$ factors must be square-free, and for $r\geq 0$ it follows that
%$c_j^{(r)}(n) = 0$
%if
%$j > \Omega(n).$
%\par
with the special case
$j = -r,$
\begin{align}\nonumber
 c_j^{(-j)}(n) &= (-1)^j \sum_{n_1 n_2 \cdots n_j = n} (\mu-e)(n_1)\, (\mu-e)(n_2) \cdots (\mu-e)(n_j)\\
 &=\sum_{n_1 n_2 \cdots n_j = n} (e-\mu)(n_1)\, (e-\mu)(n_2) \cdots (e-\mu)(n_j)\nonumber \\
 &= (e-\mu)^{{*j}}(n) \qquad (n \in {\mathbb N}),
\label{especial}\end{align}
having the interpretation as
$(-1)^{\Omega(n) + j}$
times the number of ordered factorisations of
$n$
into
$j$
non-trivial, square-free factors.

With the Dirichlet convolution notation established for the associated divisor function $c_j^{(r)}(n)$,
we next require a lemma.
\begin{lemma}\label{revsum}
Let
$(a_j)_{j \in \mathbb{N}_0}$
and
$(b_j)_{j \in \mathbb{N}_0}$
be number sequences. If
\[
 a_j = \sum_{i=0}^j \binom{j}{i} b_i \qquad (j \in \mathbb{N}_0),
\]
then
\[
 b_j = \sum_{i=0}^j (-1)^{j-i} \binom{j}{i} a_i \qquad (j \in \mathbb{N}_0),
\]
and vice versa.
\end{lemma}
\begin{proof}
For any
$ j \in \mathbb{N}_0$, we have
\begin{align*}
 \sum_{i=0}^j (-1)^{j-i} \binom{j}{i} \sum_{\ell=0}^i \binom{i}{\ell} b_\ell
= &\sum_{\ell=0}^j \left(\sum_{i=\ell}^j (-1)^{j-i} \binom{j}{i} \binom{i}{\ell} \right) b_\ell\\
=& \sum_{\ell=0}^j \left(\sum_{k=0}^{j-\ell} (-1)^k \binom{j}{j-k} \binom{j-k}{\ell}\right) b_\ell
\end{align*}
after the change of variables
$k = j - i$.
The claimed formula now follows by observing that
\begin{align*}
 \sum_{k=0}^{j-\ell} (-1)^k \binom{j}{j-k} \binom{j-\ell}{\ell}
& = \sum_{k=0}^{j-\ell} (-1)^k \frac{j!\,(j-k)!}{k!\,(j-k)!\,l!\,(j-k-\ell)!}\\
= \binom{j}{\ell} \sum_{k=0}^{j-\ell} (-1)^k \binom{j-\ell}{k}
 &= \binom{j}{\ell} (1-1)^{j-\ell} = \delta_{j,\ell}
 = \begin{cases}
0 &\text{if } j \neq \ell,   \\
1 &\text{if } j=\ell.   \end{cases}
\end{align*}
The converse follows by an almost identical calculation.
\end{proof}
\begin{example}
As a sample application, using the above lemma on Lemma 4 of \cite{mcl2}, which states that
\[
 c_j^{(r)} = \sum_{i=0}^r \binom{r}{i} c_{j+i} \qquad (r \in \mathbb{N}_0; j \in \mathbb{N}),
\]
we obtain the following expression of the non-trivial divisor
function in terms of associated divisor functions
\[
 c_{j+i} = \sum_{r=0}^i (-i)^{i-r} \binom{i}{r} c_j^{(r)} \qquad (j \in \mathbb{N}).
\]
\end{example}

\begin{proof}[Proof of Theorem \ref{enum22}]
We bear in mind the definition of the Dirichlet convolution given in (\ref{econvo}), and sum over all possible $m$-tuples \[
a=(n_1,\ldots,n_m)\in \mathbb{N}_{\geq 2}^m,
\quad\text{with}\quad n=n_1\times\ldots \times n_m\in \mathbb{N}, \quad\text{and}\quad \ell=(\ell_1,\ldots , \ell_m)\in \mathbb{N}^m,
\] where $|\ell|=\ell_1+\ldots +\ell_m$. Applying the formula given in (\ref{eq:tuplecount}), which counts the number of $m$-part joint ordered factorisations of $n$, for each individual
$m$-tuple $a$, we have that
\[
 {\cal N}_m(n) = \sum_{\substack{a \in \mathbb{N}_{\geq 2}^m \\ \prod n_j = n}} M_a(n)
 = \sum_{\substack{a \in \mathbb{N}_{\geq 2}^m\ \\ \prod n_j = n}}\,\, \sum_{\ell \,\,\in \mathbb{N}^m} \binom{|\ell|}{\ell} \prod_{j=1}^m (e-\mu)^{*\ell_j} (n_j)
\]
\[
= \sum_{\ell \,\,\in \mathbb{N}^m}\binom{|\ell|}{\ell} \,\, \sum_{\substack{a \in \mathbb{N}_{\geq 2}^m\ \\ \prod n_j = n}} \prod_{j=1}^m (e-\mu)^{*\ell_j} (n_j)
 = \sum_{\ell\,\, \in \mathbb{N}^m}\binom{|\ell|}{\ell} \left( \mathop{*}_{j=1}^m (e-\mu)^{*\ell_j} \right)(n)
 \]
\be\label{eqmue}
 = \sum_{\ell \,\,\in \mathbb{N}^m} \binom{|\ell|}{\ell} (e-\mu)^{*|\ell|}(n)
 = \sum_{\ell \,\,\in \mathbb{N}^m} \binom{|\ell|}{\ell}  c_{|\ell|}^{-|\ell|}(n),
\ee
where we have used equation (\ref{especial}) in the final step.
For the integer
$n$,
we then have the integer sequence
 $({\cal N}_m(n))_{m \in \mathbb{N}}$, which takes non-zero values for $1\leq m\leq \Omega(n)$.

Additionally, for ${m \in \mathbb{N}}$ we define the function $\tilde{\cal N}_m(n)$, similarly to (\ref{eqmue}), but with
the $m$-tuple $\ell$ running over $\mathbb{N}_0^m$, so that
%\[
% \tilde{\cal N}_m(n) = \sum_{\ell\,\, \in \mathbb{N}_0^m} \binom{|\ell|}{\ell} (e-\mu)^{*|\ell|}(n)
% = \sum_{L=0}^\infty\left(\sum_{\ell\,\, \in \mathbb{N}_0^m, |l| = L}\binom{L}{\ell}\,1^{\ell_1}\ldots 1^{\ell_m}\,\right) %(e-\mu)^{*L}(n)
% \]
 \begin{align}
\tilde{\mathcal{N}}_m(n) = \sum_{\ell \in \mathbb{N}_0^m} \binom{|\ell|}{\ell} c_{|\ell|}^{-|\ell|}(n) = \sum_{L=0}^{\Omega(N)}\bigg(c_{L}^{-L}(n)\sum_{\substack{\ell\in \mathbb{N}_0^m\\ |\ell| = L}}\binom{L}{\ell}1^{\ell_1}\dots 1^{\ell_m}\,\bigg) =\sum_{L=0}^{\Omega(n)} m^L c_{L}^{-L}(n),\label{eqmpl}
\end{align}
where in the last equality we have used $c_L^{-L}(n)=0$ for $L>\Omega(n)$, and also
\[
\sum_{\substack{\ell\in \mathbb{N}_0^m\\ |\ell| = L}}\binom{L}{\ell}1^{\ell_1}\dots 1^{\ell_m}\,=
\underbrace{(1+1+\ldots +1)^L}_{m\,\,\,\text{times}} = m^L.
\]
The expression for $\tilde{\mathcal{N}}_m(n)$ in (\ref{eqmpl}) can be expanded over $\ell\in\mathbb{N}_0^{m-r}$,
by counting the number of choices for $r$ entries to be 0 in the $m$-tuple for $\ell$. Thus we obtain
\begin{align*}
\tilde{\mathcal{N}}_m(n) =&\: \binom{m}{0}\sum_{\ell  \in \mathbb{N}^m} \binom{|\ell|}{\ell}c_{|\ell|}^{-|\ell|}(n)
 +  \binom{m}{1}\sum_{\ell  \in \mathbb{N}^{m-1}} \binom{|\ell|}{\ell} c_{|\ell|}^{-|\ell|}(n)\\
 +&  \binom{m}{2}\sum_{\ell \,\, \in \mathbb{N}^{m-2}} \binom{|\ell|}{\ell} c_{|\ell|}^{-|\ell|}(n)
 + \cdots +  \binom{m}{m}\sum_{\ell  \in \mathbb{N}^0}  \binom{|\ell|}{\ell}c_{|\ell|}^{-|\ell|}(n)\\
=&\: \sum_{k=0}^m \binom{m}{m-k}\mathcal{N}_k(n)= \sum_{k=0}^m \binom{m}{k}\mathcal{N}_k(n),
\end{align*}
by equation (\ref{eqmue}).

Now $0^L=0$ for $L\geq 1$ and $c_0^{-0}(n)=0$ for all $n\in\mathbb{N}_{\geq 2}$,
so we have that $\tilde{\mathcal{N}}_0(n)=0$ for all $n\in\mathbb{N}_{\geq 2}$.
Hence we need only consider the sum from $k=1$, and applying Lemma \ref{revsum} gives us
\begin{align*}
\mathcal{N}_m(n) = \sum_{k=1}^m (-1)^{m-k} \binom{m}{k} \tilde{\mathcal{N}}_k(n)
 = \sum_{L=0}^{\Omega(n)}  \left(\sum_{k=1}^m (-1)^{m-k}  k^L \binom{m}{k} \right) c_{L}^{-L}(n).
\end{align*}
Using the the identity for the Stirling numbers of the second kind (cf. \cite{comtet} Theorem 5.1.A)
\begin{align*}
S(L,m) = \sum_{k=1}^m (-1)^{m-k} k^{L-1}\frac{1}{(k-1)! (m-k)!},
\end{align*}
we can write
\begin{align*}
\mathcal{N}_m(n) =&\: \sum_{L=0}^{\Omega(n)}  \left(\sum_{k=1}^m (-1)^{m-k}  k^L \binom{m}{k} \right) c_{L}^{-L}(n)\\
=&\: \sum_{L=0}^{\Omega(n)}  \left(\sum_{k=1}^m (-1)^{m-k} k^{L}\frac{m!}{k! (m-k)!} \right) c_{L}^{-L}(n)\\
=&\: m!\sum_{L=0}^{\Omega(n)}  \left(\sum_{k=1}^m (-1)^{m-k} k^{L-1}\frac{1}{(k-1)! (m-k)!} \right) c_{L}^{-L}(n)\\
= &\sum_{L=0}^{\Omega{(n)}} m!\, S(L,m)\, (e-\mu)^{*L}(n)
 = \sum_{L=0}^{\Omega{(n)}} m!\, S(L,m)\, c_L^{(-L)}(n).
%=&\: m!\sum_{L=0}^{\Omega(n)}  S(L,m) c_{L}^{-L}(n),
\end{align*}
\end{proof}

%In the next section we derive sums over divisor relations for our counting function ${\cal N}_2(N)$.

%\section{Sums over divisors and proof of Theorem \ref{sdfl}}
%In \cite{long} the complementing set system
%\begin{equation}\label{sdflong}
%C=A+B=\{x|x=a+b,\,\,a\in A,\,\,B\in B\}
%\end{equation}
%was considered for all complementing subsets of $\langle N\rangle= \{0,1,2\ldots, N-1\}$. In Theorem 2 they state that the %number $C(N)$ of complementing subsets of $\langle N\rangle$, is given by
%\[
%C(N)=\mathop{\sum_{d\mid N}}_{d<N} C(d).
%\]
%\begin{remark}
%The paper \cite{long} would have benefitted from some examples to demonstrate how this sum over divisors relates to the %above tables given in Example \ref{ex3}. To help clarify this matter we now develop the sum over divisors relation for our %counting function ${\cal N}_m(N)$, as stated below in Theorem \ref{sdf}
%\end{remark}

\begin{proof}[Proof of Theorem \ref{sdf}]
We begin with Theorem 1(a) of \cite{mcl4}, which gives a three term recurrence relation for the associated divisor function, stated here as
\[
c_{k+1}^{(r)}=c_k^{(r+1)}-c_k^{(r)},
\]
and setting $r=-L$ and $k=L-1$, we obtain
\[
c_{L}^{(-L)}=c_{L-1}^{(-(L-1))}-c_{L-1}^{(-L)}.
%=c_{j-1}^{(-(j-1))}-\sum_{d\mid n}c_{j-1}^{(-(j-1))}.
\]
Substituting in the formula for ${\cal N}_m(n)$ given in Theorem~\ref{enum22}, where
the upper index of summation $\Omega{(n)}$ is formally replaced by $\infty$ as $c_{L}^{(-L)}(n)=0$ for $n\geq \Omega{(n)}$, we have that
\[
{\cal N}_m(n) = \sum_{L=0}^\infty m!\, S(L,m)\, c_{L}^{(-L)}(n)
=\sum_{L=0}^\infty m!\, S(L,m)\,\left ( c_{L-1}^{(-(L-1))}(n)-c_{L-1}^{(-L)}(n)\right ).
\]
We now substitute for the well known three-term recurrence relation for the Stirling numbers of the second kind \cite{riordan}
$
S(L,m)=m\,S(L-1,m)+S(L-1,m-1),
$
along with the fact that $S(0,m)=0$ for $m\geq 1$ to obtain
\[
{\cal N}_m(n)
=\sum_{L=1}^\infty m!\, \left (m\,S(L-1,m)+S(L-1,m-1)\right )\,
\left ( c_{L-1}^{(-(L-1))}(n)-c_{L-1}^{(-L))}(n)\right ).
\]
Lowering the index of summation to start at $L=0$ gives us
\[
{\cal N}_m(n)
=\sum_{L=0}^\infty m!\, \left (m\,S(L,m)+S(L,m-1)\right )\,
\left (c_{L}^{(-L)}(n)-c_{L}^{(-L-1)}(n)\right ).
\]
Expanding out the brackets and rearranging, we have four summation terms which by Theorem \ref{enum22} can be written as
% (\ref{sdf1})
\[
{\cal N}_m(n)= m\left ({\cal N}_m(n)+{\cal N}_{m-1}(n)\right )-\sum_{L=0}^\infty m!\, \left (m\,S(L,m)+S(L,m-1)\right )
c_{L}^{(-L-1)}(n).
\]
In the final step we sum each side over the divisors $d$ of $n$, noting from (\ref{sodd}) that $({c_j^{(r)}}{*}1)(n)=c_j^{(r+1)}(n)$, to obtain
\begin{align*}
\sum_{d\mid n} {\cal N}_m(d)&=\left (m{\sum_{d\mid n}}\left ( {\cal N}_m(d)+{\cal N}_{m-1}(d) \right )\right )
-\sum_{L=0}^\infty m!\, \left (m\,S(L,m)+S(L,m-1)\right )
c_{L}^{(-L)}(n)\\
&=\left ( m{\sum_{d\mid n}}\left ( {\cal N}_m(d)+{\cal N}_{m-1}(d) \right )\right )
-m\left ( {\cal N}_m(n)+{\cal N}_{m-1}(n) \right )\\
&= \mathop{m\sum_{d\mid n}}_{\,\,\,\,\,d<n}\left ( {\cal N}_m(d)+{\cal N}_{m-1}(d) \right ),
\end{align*}
where we have again used Theorem \ref{enum22}. Rearranging then gives the desired result
\[
{\cal N}_m(n)=\mathop{\sum_{d\mid n}}_{d<n}\left ((m-1) {\cal N}_m(d)+m{\cal N}_{m-1}(d) \right ).
\]
For the second identity (\ref{sdf2}), we consider the sum over all divisors of $n$, written as
\[
m\left ( {\cal N}_m(n)+{\cal N}_{m-1}(n)\right )
=\sum_{d\mid n}\left ((m-1) {\cal N}_m(d)+m{\cal N}_{m-1}(d) \right ),
\]
and then apply M\"obius inversion to obtain
\[
(m-1) {\cal N}_m(n)+m{\cal N}_{m-1}(n)=
m\sum_{d\mid n}\mu\left( \frac{n}{d}\right)\left ({\cal N}_m(d)+{\cal N}_{m-1}(d)\right ).
\]
Noting that $\mu(n/n)=\mu(1)=1$, and rearranging, we obtain the statement (\ref{sdf2})
\[
{\cal N}_m(n)=
 -m \mathop{\sum_{d\mid n}}_{d<n}\mu\left (\frac{n}{d}\right)\left ({\cal N}_m(d)+{\cal N}_{m-1}(d) \right ).
\]

To see (\ref{sdf3}), equation (\ref{sdf1}) with $m=2$ yields
\[
{\cal N}_2(n)=
\mathop{\sum_{d\mid n}}_{d<n}\left ( {\cal N}_2(d)+2{\cal N}_{1}(d) \right ).
\]
Here ${\cal N}_{1}(1)=0$ and ${\cal N}_{1}(d)=1$ for $d>1$, with $d_2(n)$ the well known divisor function, which counts the number of ways of writing $n=n_1\times n_2$, with $n_1,n_2\geq 1$. Removing the two cases $n_1=1$ and $n_2=1$, we obtain the non-trivial divisor function $c_2(n)=d_2(n)-2$.

\end{proof}

%two-dimensional considerations of Long \cite{long}, the correct form of which was provided by MacMahon in \cite{macmahon}
%(see p5 of \cite{survey}).

\begin{remark}
The sum over divisors formula (\ref{sdf3}) of Theorem \ref{thm55} for the 2-part joint ordered factorisation case ${\cal N}_2(n)$ was also stated by MacMahon \cite{macmahon} using alternative notation.  In \cite{long} C. T. Long made the connection between the number of two-dimensional additive (complementing) set systems and ${\cal N}_2(n)$ (referred to as $C(n)$ in \cite{long}). However, as stated by Knopfmacher \cite{survey}, MacMahon's sum over divisors formula was the correct version, as the formula given by Long omitted (in our notation) the term~$2c_2(n)$.

The ideas of Long can be extended to $m$ parts, where the joint ordered factorisations underpin the additive set system constructions. As an illustration for the use of joint ordered factorisations, we show in the following section how the counting functions ${\cal{N}}_m(n)$ can be used to count the number of different $m$ part sum systems for a given positive integer $n\geq 2$.

%as an illustrative application, we now show how each joint ordered factorisation of $n$ can be used to uniquely construct %and so count the number of distinct $m$ part sum systems.

\end{remark}

\noindent
\section{Sum systems}

Sum systems are the third key topic under consideration here, which we now describe before stating our results which include algebraic invariance properties for \emph{centred sum systems}.

In what follows, for $a \in {\mathbb N},$ we use the notation $\langle a \rangle= \{0, 1, \dots, a-1\}$  so the $a$-term arithmetic progression with start value $r,$ step size $s$ can be expressed as
$s \langle a \rangle + r \ = \{r, r+s,  \dots, r+(a-1)s \}.$

An $m$-dimensional additive set system for a given target set of integers, $T$, is a collection of $m$ sets of integers, $A_1,A_2,\ldots A_m$, each with cardinality $|A_j|$, for $ j\in\{1,\dots,m\}$, such that their sumset satisfies
\[
T=\sum_{j=1}^m A_j=\bigg\{ \sum_{j=1}^ma_j\mid a_j\in A_j \bigg\}.
\]
The cardinality of these sets satisfy the equation $|T|=|A_1||A_2|\ldots |A_m|$ if and only if each element $t\in T$ is represented by a unique sum in this set sum, and we then call the additive set system a \textit{basis} for $T$.

The study of additive systems dates back to de Bruijn's paper \cite{bruijn} on (possibly infinite) sets of non-negative integers $A_k$, with $|A_k|\neq 0$ and $0\in A_k$, for $T=\N_0$. He referred to such an additive system as a \emph{number system}.

Subsequently such systems became known as complementing set systems \cite{vaidya,long}, the latter paper focusing on systems that uniquely represents the first $n$ consecutive integers, $T=\{0,1,2,\ldots,n-1\}$, and enumerating the systems when $m=2$. %More recently an exhaustive construction of complementing set systems, called \textit{sum systems}, based on integer %factorisations of the cardinalities $|A_1|,\dots,|A_m|$, with general $m\in\N$, was given by \cite{mnh1}, and the number %of sum systems was enumerated in the general case \cite{mcl4}.

A question posed by de Bruijn in the closing remarks of \cite{bruijn} refers to an earlier publication of his \cite{bruijn1}, concerning the analogous problem for number systems representing uniquely all integers in $T=\Z$. Of this de Bruijn says ``That problem is much more difficult that the one dealt with above, and it is still far from a complete solution.''
It is this more general consideration of de Bruijn that also partially motivates the results in this section, building on the related concepts of \emph{sum-and-distance systems}, established in \cite{mcl2,hillthesis,mnh1,law}.

We now formally define \emph{sum systems}, which restrict the target set to be
the first $n$ non-negative integers $\langle n \rangle$, and require uniqueness of the representation of each number.

\begin{definition}\label{maindefn}
Let $m\in\mathbb{N}$, $A_j \subset {\mathbb N_0}$ $(j \in \{1, \dots, m\})$, with cardinality  $|A_j|~=~n_j$, and $\prod_{j=1}^m|A_j|=n=n_1\times n_2\times \ldots\times n_m$. Then we call $A_1, A_2, \dots, A_m$ an \textit{$m$-part sum system} if
\begin{equation}
 \sum_{j=1}^m A_j = \left\langle{n}\right\rangle.
\end{equation}
\end{definition}
\begin{lemma}
Let $m \in {\mathbb N},$ and
suppose the sets
$A_1, A_2, \dots, A_m \subset {\mathbb N}_0$
form an $m$-part sum system for the first $n$ non-negative integers.
Then the \emph{centred sum system}
$C_1, C_2, \dots, C_m$, defined such that
$$ \label{consec}C_j = A_j-\frac{(\max A_j)}{2},$$
generates under set addition
\[
 \sum_{j=1}^m C_j = \left\{-\frac{n-1}{2},-\frac{n-3}{2},\ldots\ldots,\frac{n-3}{2}, \frac{n-1}{2}\right\},
\]
i.e. either $n$ consecutive integers centred about 0
or $n$ consecutive half-integers centred about 0, depending on whether $n$ is an odd or an
even integer.
\end{lemma}
\begin{proof}
From Lemma 3.2 and Theorem 3.3 of \cite{mnh1}, every sum system component set is palindromic, satisfying $A_j= (\max A_j) - A_j$. As sum system component sets $A_j$, contain non-negative integers and always include $0$, it follows from the above definition that each centred sum system component set is centred about the origin, has cardinality $|C_j|=|A_j|$, with max$\,C_j=(\text{max}\,A_j)/2$, and so we have
\[
\sum_{j=1}^m \text{max}\,C_j=\frac{1}{2}\sum_{j=1}^m \text{max}\,A_j=\frac{n-1}{2}.
\]
Hence when $n$ is an odd integer, $(n-1)/2$ is also an integer and $C_1+\ldots +C_m$ generates the set of consecutive integers between $-(n-1)/2$ and $(n-1)/2$. Conversely, when $n$ is an even integer we find that our centred sum system
under set addition generates the consecutive half-integers in this interval.
\end{proof}
An explicit construction method for all sum systems was proven in \cite[Theorem~6.7]{mnh1}, whereby a bijection between the set of joint ordered factorisations for an $m$-tuple of natural numbers $a=(n_1, \dots, n_m)$, and the set of all sum systems for the corresponding $m$-part sum system was established, as stated below in Proposition~\ref{tssbuild}.
%that,
%given a joint ordered factorisation, the sets
%form a sum system, and that conversely any sum system arises from some joint ordered factorisation of the %cardinalities %of its component sets in this way, thus establishing a bijection between joint ordered factorisations %and sum systems.
\begin{proposition}\label{tssbuild}
Let
$m \in {\mathbb N}.$
Suppose the sets
$A_1, A_2, \dots, A_m \subset {\mathbb N}_0$
%where
% min(A_j smo) < min(A_{j+1} smo)
% (j in \{1, dots, m-1\}),
form a sum system.
Let
$n_j := |A_j|$
$(j \in \{1, \dots, m\}).$

Then there is a joint ordered factorisation
$((j_1, f_1), \dots, (j_L, f_L))$
of
$(n_1, \dots, n_m)$
such that
\begin{align}
 A_j &= \sum_{j_\ell = j} \left(\prod_{s=1}^{\ell-1} f_s \right) \langle{f_\ell}\rangle
  = \sum_{j_\ell = j} F(\ell) \langle f_\ell\rangle
 \qquad (j \in \{1, \dots, m\}).
\label{essbuild}\end{align}
Conversely, given any joint ordered factorisation of
$n$ over the $m$-tuple $a=(n_1,\ldots,n_m),$
the construction $(\ref{essbuild})$
generates uniquely an $m$-part sum system.
\end{proposition}
\begin{corollary}
The number of $m$ part sum systems and centred sum systems for the positive integer $n\geq 2$, is given by ${\cal N}_m(n)$,
the counting function for the number of $m$ part joint ordered factorisations for $n$.
%\begin{proof}
%The proof follows immediately from the bijection between the set of all $m$ part joint ordered factorisations for
%the integer $n$ over the tuple $a=(n_1, \dots, n_m)$, and the set of all $m$ part sum systems for $n$
%established \cite[Theorem 6.7]{mnh1}.
%\end{proof}
\end{corollary}

\begin{example}\label{ex1}
Let $n=270$, and consider the corresponding 3-part sum system for $270$ over the 3-tuple $a=(n_1, n_2, n_3)=(9,5,6)$,
with joint ordered factorisation $\left ((1,3),(3,3),(1,3),(3,2),(2,5)\right )$.
Applying Proposition \ref{tssbuild} we obtain
\[
A_1=\{0, 1, 2, 9, 10, 11, 18, 19, 20\},\,\,\,
A_2=\{0, 54, 108, 162, 216\},\,\,\,
A_3=\{0, 3, 6, 27, 30, 33\},
\]
so that $|A_1|=9$, $|A_2|=5$, and $|A_3|=6$, and
$A_1+A_2+A_3=\{0,1,2,\ldots 269\}$, where $269=(9\times 5\times 6)-1$. The corresponding centred sum system components are therefore
\[
C_1=\{-10, -9, -8, -1, 0, 1, 8, 9, 10\},\quad
C_2=\{-108, -54, 0, 54, 108\},
\]
and $C_3=\frac{1}{2}\{-33,-27, -21, 21, 27, 33\}$.
Hence we have
\[
C_1+C_2+C_3
=\frac{1}{2}\{-269,-267,\ldots -1,1,\ldots 267,269\}.
\]
The first $n=270$ (odd) half integers centred about the origin.
\end{example}

\begin{remark}
In \cite{mnh1} sum-and-distance systems were classified into two types; \textit{inclusive} and \textit{non-inclusive}, where inclusive sum-and-distance systems generate central-symmetric sets around zero of consecutive integers, and non-inclusive systems generated a central-symmetric set around zero of consecutive odd integers. The terminology of \textit{non-inclusive} or \textit{inclusive}, corresponds respectively to whether the target set is generated solely by the sums and differences of elements between the different sets or whether the elements of the individual sets themselves are also required.

Here we work with the single concept of centred sum systems, removing the need to distinguish between inclusive and non-inclusive sum and distance systems. This approach leads to our main (invariance) result for centred sum systems with fixed
values of $n$ and $m$, but allowing distinct joint ordered factorisation constructions in Proposition \ref{tssbuild}, the numerical invariance properties stated below.
\end{remark}

\begin{theorem}\label{invariance}
Let $n\in \mathbb{N}$ with corresponding $m$-tuple $a=(n_1,n_2,\ldots, n_m)$, and let $A_1,\ldots A_m$ be a an $m$-part sum system for $n$, so that $\prod_{j=1}^m|A_j|=\prod_{j=1}^m n_j=n$, and $\sum_{j=1}^m A_j=\langle n \rangle$. Furthermore let $C_1,C_2,\ldots C_m$ be the centred sum system derived from the $A_j$ by subtracting
$\frac{1}{2}\max\,(A_j)$ from each component set $A_j$. Finally, for the sum systems and centred sum systems as given above, define the respective sums of elements $\sigma_A(n)$, and sums of squares of elements $\tau_{C}(n)$, such that
\begin{equation}
\sigma_{A}(n)=\sum_{j=1}^m\frac{n}{n_j}\sum_{a\in A_j}a\, ,\quad\text{and}\quad
\tau_{C}(n)=\sum_{j=1}^m\frac{n}{n_j}\sum_{c\in C_j}c^2.
\end{equation}
Then we have
\[
\sigma_A(n)=\frac{n(n-1)}{2}=\sum_{k=1}^{n-1}k=\binom{n}{2},\quad\text{and}\quad
\tau_{C}(n)=\frac{n(n^2-1)}{12}=\frac{1}{2}\binom{n+1}{3}.
\]
\end{theorem}
\begin{remark}
Any vector composed of the elements from all component sets of a centred sum system can therefore be thought of as an integer or half-integer lattice point on an ellipsoid in $n$-dimensional Euclidean space. Here the normalisation factors $n/n_j$ in Theorem \ref{invariance} transform the points on the ellipsoid into points on an $n$-dimensional sphere with radius $\sqrt{\tau_C(n)}$. This sphere only depends on $n$ and all different multifactorisations of $n$ correspond to integer or half-integer points on it.
\end{remark} 
\begin{example}\label{ex2}
For the 3-part sum system given in Example \ref{ex1}, we have that $n=270=9\times 5\times 6$, and we find that
\[
\sigma_A(270)=36\,315=\frac{270\times 269}{2},\quad\text{and}\quad
\tau_{C}(270)=\frac{3\,280\,455}{2}=\frac{271\times 270\times 269}{12},
\]
as shown in Theorem \ref{invariance}. Similarly, for $n=270$ and $a=(n_1,n_2,n_3)=(9,5,6)$, but using instead the
joint ordered factorisation $\left ((1,3),(2,5),(3,3),(1,3),(3,2)\right )$ to obtain the distinct centred sum system
\[
C_1=\{-46, -45, -44, -1, 0, 1, 44, 45, 46\},\quad C_2=\{-6, -3, 0, 3, 6\}
\]
and $C_3=\frac{1}{2}\{-165, -135, -105, 105, 135, 165\}$,
we again find that
$$\tau_{C}(270)=\frac{3\,280\,455}{2},$$
though the sums of squares over the individual component sets are different.
\end{example}
%To give the reader a structural overview of this present work
%\section{More on centred sum systems and proof of Theorem \ref{invariance}}

To prove Theorem \ref{invariance} we first require a lemma.
\begin{lemma}\label{addingsets}
Let $A$ and $B$ be two disjoint sets of numbers. Furthermore let $B$ be symmetric around the origin, i.e. $B=-B$. Then we have
\begin{equation*}
    \sum_{c\in(A+B)}c=|B|\sum_{a\in A}a,\quad\text{and}\quad
    \sum_{c\in(A+B)}c^2=|B|\sum_{a\in A}a^2+|A|\sum_{b\in B}b^2.
\end{equation*}
\end{lemma}

\begin{proof}
Writing $A=\{a_1,a_2,\ldots, a_{|A|}\}$, we have $A+B=\{a_1+B,a_2+B,\dots,a_{|A|}+B\}$, and taking the sum over all elements in $A+B$, in conjunction with $\sum_{b\in B}b=0$, gives us
\begin{align*}
    \sum_{c\in(A+B)}c&=(a_1+B)+\dots+(a_{|A|}+B)=\sum_{i=1}^{|A|}\sum_{j=1}^{|B|}(a_i+b_j)\\
    &=|B|\sum_{i=1}^{|A|}a_i+|A|\sum_{j=1}^{|B|}b_j=|B|\sum_{a\in A}a,
\end{align*}
and for the sum of elements squared, writing $B=\{b_1,b_2,\ldots, b_{|B|}\}$, we have
\begin{equation*}\begin{split}
    \sum_{c\in(A+B)}c^2=&(a_1+B)^2+\dots+(a_{|A|}+B)^2=\sum_{i=1}^{|A|}\sum_{j=1}^{|B|}(a_i^2+a_ib_j+b_j^2)\\
    =&|B|\sum_{i=1}^{|A|}a_i^2+\left (\sum_{i=1}^{|A|}a_i\right)\left (\sum_{j=1}^{|B|}b_j\right)+|A|\sum_{j=1}^{|B|}b_j^2=|B|\sum_{a\in A}a^2+|A|\sum_{b\in B}b^2,
\end{split}
\end{equation*} as the sum of the cross-terms is zero.
\end{proof}

\begin{proof}[Proof of Theorem \ref{invariance}]

Regarding $\sigma_A(n)$, we know from Theorem 3.2 of \cite{mnh1} that each component set $A_j$ of a sum system is palindromic, centred about $(\text{max}\, A_j)/2$, so on average each element is of this magnitude. The sum of the elements of each component set is therefore given by
$$\sum_{a\in A_j}a =\frac{|A_j|(\text{max} A_j)}{2}=\frac{n_j(\text{max} A_j)}{2}.$$
Hence
\[
\frac{1}{n_j}\sum_{a\in A_j}a =\frac{(\text{max} A_j)}{2},
\]
and as the sum over the maximal element of each component set $A_j$ is $n-1$, multiplying through by $n$ and taking the sum over all component sets of the sum system gives us
\[
n\sum_{j=1}^m \frac{1}{n_j}\sum_{a\in A_j}a=\frac{n}{2}\sum_{j=1}^m (\text{max}\, A_j)=\frac{n(n-1)}{2}.
\]
For the sum of squares function $\tau_C(n)$, by repeated application of Lemma \ref{addingsets} we obtain
\[
\sum_{j=1}^m\sum_{c\in C_j}c^2=\sum_{j=1}^m\left (\mathop{\prod_{k=1}^m}_{k\neq j}|C_k|\right )\sum_{c\in C_j}c^2
=\sum_{j=1}^m \frac{n}{n_j}\sum_{c\in C_j}c^2.
\]
When at least one of the centred sum component sets has even cardinality, the sum of squares on the left-hand-side is simply twice the sum of squares of all the consecutive half-integers between $\frac{1}{2}$ and $\frac{n-1}{2}$, and so we have
\[
2 \sum _{j=1}^{\frac{n}{2}} \left(\frac{1}{2} (2 j-1)\right)^2
=\frac{1}{12} \left(n^3- n\right)=\frac{1}{12} (n-1)n (n+1).
\]
If all of the centred sum component sets have odd cardinality, then $n=n_1 n_2 \ldots n_m$ is odd and the sum of squares on the left-hand-side is simply twice the sum of squares of all the integers between $1$ and $(n-1)/2$, and we find that
\[
2 \sum _{j=1}^{\frac{n-1}{2}} j^2=\frac{1}{12} (n-1)n (n+1)=\frac{1}{2}\binom{n+1}{3}.
\]
\end{proof}

% In \cite[Section 3]{mnh1} a bijection was established between sum systems and sum-and-distance systems.
%Stepwise, the proof (via laurent polynomials) establishes the property that each sum system component set $A_j$ is %palindromic, whereby
%$$ A_j = (\max A_j) - A_j,$$ i.e. $x \in A_j$ if and only if $(\max A_j - x) \in A_j,$ and $A_j-(\max A_j)/2$ is %\emph{centred} about the zero. The two cases of $n$ being even or odd are considered separately to establish a %bijective construction, corresponding to (\ref{eqhybrid}) given in Definition \ref{maindefn}. The two bijections thus %establish that every (consecutive) sum-and-distance system can be uniquely associated with a joint ordered %factorisation and conversely.

\section{Conclusion and further areas for consideration}
We have established here the formula in Theorem \ref{enum22} for the function ${\cal N}_m(n)$, which counts all $m$-part joint ordered factorisations for an integer $n$. This function arises naturally when counting either
$m$ part joint ordered factorisations or $m$ part sum systems. The implicit sum over divisor relations for
${\cal N}_m(n)$, given in Theorem~\ref{sdf}, has an inherent symmetry which indicates that the number theoretic function
${\cal N}_m(n)$ might be worthy of further investigation in its own right.

For sum systems we have demonstrated that for a fixed target integer $n$, the sum system components exhibit an invariance in the sum of their elements, and for the centralised sum systems this translates into a sum of squares invariance.

In consideration of de Bruijn's question regarding complementing set systems (sum systems) for $\mathbb{Z}$ rather than $\mathbb{N}$, the results detailed in this paper, combined with those of \cite{mcl2,mnh1,mcl4} give methods for constructing centro-symmetric sets of integers or half integers about the origin, via joint ordered factorisations. For target sets which are not symmetric about the origin this is a much harder question and could form the focus of future investigations.

\subsection*{Disclosure statement} No potential conflict of interests was reported by the authors.

%The authors are grateful to the anonymous referee for constructive comments
%which helped streamline the paper, especially for suggesting the extensive use of
%convolution notation.
\subsection*{Acknowledgements} A. Law's research was supported by an EPSRC
DTP grant EP/R513003/1.

\subsection*{Appendix}

We conclude this section with some tables of numerical values for ${\cal N}_m(n)$ for $n=1,2,\ldots 32$.
We note that $\mathcal{N}_1(1)=0$, since the corresponding joint ordered factorisation is $\big((1,1)\big)$ and we require all $f$-values to be $\geq2$.

The first values of ${\cal N}_m(n)$ for $n=1,2,\ldots 32$, with $1\leq m\leq 4$, are given by

\[
\begin{array}{c|cccccccccccccccc}
n &1 & 2 & 3 & 4 & 5 & 6 & 7 & 8 & 9 & 10 & 11 & 12 & 13 & 14 & 15 & 16   \\ \hline
{\cal N}_1(n) & 0 & 1 & 1 & 1 & 1 & 1 & 1 & 1 & 1 & 1 & 1 & 1 & 1 & 1 & 1 & 1   \\
{\cal N}_2(n) &0 & 0 & 0 & 2 & 0 & 4 & 0 & 6 & 2 & 4 & 0 & 14 & 0 & 4 & 4 & 14   \\
{\cal N}_3(n) &0 & 0 & 0 & 0 & 0 & 0 & 0 & 6 & 0 & 0 & 0 & 18 & 0 & 0 & 0 & 36  \\
{\cal N}_4(n) &0 & 0 & 0 & 0 & 0 & 0 & 0 & 0 & 0 & 0 & 0 & 0 & 0 & 0 & 0 & 24   \\
\end{array}
\]
\[
\begin{array}{c|cccccccccccccccc}
n &   17 & 18 & 19 & 20 & 21 & 22 & 23 & 24 & 25 & 26 & 27 & 28 & 29 & 30 & 31 & 32 \\ \hline
{\cal N}_1(n) &  1 & 1 & 1 & 1 & 1 & 1 & 1 & 1 & 1 & 1 & 1 & 1 & 1 & 1 & 1 & 1 \\
{\cal N}_2(n) &  0 & 14 & 0 & 14 & 4 & 4 & 0 & 38 & 2 & 4 & 6 & 14 & 0 & 24 & 0 & 30 \\
{\cal N}_3(n) &  0 & 18 & 0 & 18 & 0 & 0 & 0 & 126 & 0 & 0 & 6 & 18 & 0 & 36 & 0 & 150 \\
{\cal N}_4(n) &  0 & 0 & 0 & 0 & 0 & 0 & 0 & 96 & 0 & 0 & 0 & 0 & 0 & 0 & 0 & 240 \\
\end{array}
\]
from which it can be seen that when $n=p$, a prime number, there exists only one sum system of dimension 1 corresponding to the component set $\{0,1,2,3,\ldots,p-1\}$, with joint ordered factorisation $(1,p)$. When $n=p^2$, with $m=2$, so $(n_1,n_2)=(p,p)$, we find that there exist two joint ordered factorisations $((1,p)(2,p))$ and $((2,p)(1,p))$  and no others.

In \cite{mcl2} the number of two-dimensional joint ordered factorisation $M_{(n,n)}$, corresponding to the tuple $(n_1,n_1)$, where in our notation $n=n_1^2$, was calculated to be
\begin{equation}\label{total2djofs}
\frac{1}{2}{\cal N}_2(n)=M_{(n_1,n_1)}=\sum_{j=1}^\infty \big(c_j(n_1)c_j(n_1)+c_j(n_1) c_{j+1}(n_1)\big)
=\sum_{j=1}^\infty c_j^{(1)}(n_1)c_j^{(0)}(n_1).
\end{equation}
Unlike the case $m=1$, when $n=n_1^2$ and $m=2$ there always exist at least two joint ordered factorisations.

\noindent
\small{School of Mathematics\\
Cardiff University\\
Cardiff CF24 4AG\\
UK\\
E-Mail: ambroselaw@hotmail.co.uk;\\ LettingtonMC@cardiff.ac.uk;\\SchmidtKM@cardiff.ac.uk}


\begin{thebibliography}{99}

\bibitem{bruijn1} de Bruijn, N.G., On bases for set of integers, Publicationes Mathematicae, Debrecen I, 232-242 (1950).


\bibitem{bruijn} de Bruijn, N.G., On number systems, Nieuw Arch. Wisk. {\bf 4}, 15-17 (1956).

\bibitem{comtet} Comtet, L. Advanced Combinatorics, Reidel, Dordrecht 1974.

\bibitem{mcl2} Hill, S.L., Huxley, M.N., Lettington, M.C. and Schmidt, K.M.  Some properties and applications of non-trivial divisor functions, Ramanujan J {\bf 51}, 611-628 (2020). https://doi.org/10.1007/s11139-018-0093-9

\bibitem{hillthesis} Hill, S.L., PhD Thesis: Problems Related to Number Theory; Sum-and-Distance Systems, Reversible Square Matrices and Divisor Functions, Cardiff University (2018).

\bibitem{mnh1} Huxley, M.N., Lettington, M.C. and Schmidt, K.M. On the structure of additive systems of integers, Period Math Hung {\bf 78}, 178-199 (2019). https://doi.org/10.1007/s10998-018-00275-w

\bibitem{survey} Knopfmacher, A., Mays, M. E., A survey of factorisation counting functions, Int. J. Number Theory
{\bf 1}(4)  563-581 (2005). doi: 10.1142/S1793042105000315

\bibitem{law} Law, A.D., PhD Thesis: On topics related to sum systems, Cardiff University (2023).

\bibitem{mcl4} Lettington, M.C. and Schmidt K.M. Divisor functions and the number of sum systems, Integers (20) {\bf ~A61}, 1--13 (2020).

\bibitem{long} Long, C.T., Addition Theorems for sets of Integers, Pacific Journal of Mathematics, {\bf  23}(1) 107-112 (1967).

\bibitem{macmahon} MacMahon, P,A., Memoir on the Composition of Numbers, Phil. Trans. R. S. (A) {\bf 184} 835-901 (1893).

\bibitem{rOB} Ollerenshaw, K. and Br\'ee, D., Most-perfect pandiagonal magic squares, IMA (1998).






\bibitem{rPopo}
Popovici, C.P.
O generalizare a func\c tiei lui M\"oebius,
 Acad.\ R.\ P.\ Rom\^\i ne Stud.\ Cerc.\ Mat.\
{\bf 14}
(1963), 493--499.

\bibitem{riordan} Riordan, J. An Introduction to Combinatorial Analysis. New York: Wiley, (1980).

\bibitem{rSchSp}
Schwarz, W. and Spilker, J.
Arithmetical Functions,
LMS Lecture Note Series\,\,
{\bf 184},
Cambridge University Press, Cambridge, (1994).

\bibitem{ect} Titchmarsh, E. C., The Theory of the Riemann Zeta-Function, Oxford University Press, Oxford (1951).

\bibitem{vaidya} Vaidya, A.M., On complementing sets of nonnegative integers, Mathematics Magazine, {\bf 39}(1), 43-44

%\bibitem{tit1}
%E. C. Titchmarsh, {\it The Theory of the Riemann Zeta-function.\/}
%Oxford University Press 1951

%\end{thebibliography}
% etc
\end{thebibliography}
\end{document}